\documentclass[draft]{amsart}

\usepackage[british]{babel}
\usepackage{amssymb}
\usepackage{cite}
\newcommand{\Z}{\mathbb{Z}}
\newcommand{\R}{\mathbb{R}}

\newcommand{\Ker}{\mathop\mathrm{Ker}\nolimits}
\newcommand{\coKer}{\mathop\mathrm{coKer}\nolimits}

\newcommand{\ind}{\mathop\mathrm{ind}\nolimits}
\newcommand{\sign}{\mathop\mathrm{sign}\nolimits}
\newcommand{\Iso}{\mathrm{Iso}}
\newcommand{\GL}{\mathrm{GL}}
\newcommand{\A}{\mathcal{A}}
\newcommand{\M}{\mathcal{M}}
\newcommand{\N}{\mathcal{N}}
\newcommand{\K}{\mathcal{K}}
\newcommand{\e}{\varepsilon}
\newcommand{\s}{\sigma}
\newcommand{\per}{\!\times\!}
\newcommand{\pallino}{\smallskip $\bullet$\;}
\newcommand{\F}{\mathcal{F}}
\newcommand{\g}{\gamma}
\newcommand{\G}{\Gamma}
\renewcommand{\L}{\mathcal{L}}
\renewcommand{\S}{\mathrm{S}}

\renewcommand{\Im}{\mathop\mathrm{Img}\nolimits}

\renewcommand{\l}{\lambda}

\renewcommand{\a}{\alpha}
\renewcommand{\b}{\beta}
\renewcommand{\o}{\omega}
\renewcommand{\O}{\Omega}
\renewcommand{\t}{\theta}

\theoremstyle{plain}
\newtheorem{theorem}{Theorem}[section]
\newtheorem{corollary}[theorem]{Corollary}

\newtheorem{proposition}[theorem]{Proposition}
\newtheorem{conjecture}[theorem]{Conjecture}

\newtheorem{remark}[theorem]{Remark}
\theoremstyle{definition}
\newtheorem{notation}[theorem]{Notation}
\newtheorem{definition}[theorem]{Definition}
\newtheorem{example}[theorem]{Example}

\numberwithin{equation}{section}

\begin{document}
\title[Global persistence of the unit eigenvectors]{Global persistence of the unit eigenvectors of perturbed eigenvalue problems in Hilbert spaces: the odd multiplicity case}

\author[P.\ Benevieri]{Pierluigi Benevieri}
\author[A.\ Calamai]{Alessandro Calamai}
\author[M.\ Furi]{Massimo Furi}
\author[M.P.\ Pera]{Maria Patrizia Pera}

\thanks{The first, second and fourth authors are members of the Gruppo Nazionale per l'Analisi Mate\-ma\-tica, la Probabilit\`a e le loro Applicazioni (GNAMPA) of the Istituto Nazionale di Alta Mate\-ma\-tica (INdAM)}
\thanks{A.\ Calamai is partially supported by GNAMPA\ - INdAM (Italy)}

\date{\today}

\address{Pierluigi Benevieri -
Instituto de Matem\'atica e Estat\'istica,
Universidade de S\~ao Paulo,
Rua do Mat\~ao 1010,
S\~ao Paulo - SP - Brasil - CEP 05508-090 -
 {\it E-mail address: \tt
pluigi@ime.usp.br}}
\address{Alessandro Calamai -
Dipartimento di Ingegneria Civile, Edile e Architettura,
Universit\`a Politecnica delle Marche,
Via Brecce Bianche,
I-60131 Ancona, Italy -
 {\it E-mail address: \tt
calamai@dipmat.univpm.it}}
\address{Massimo Furi - Dipartimento di Matematica e Informatica ``Ulisse Dini'',
Uni\-ver\-sit\`a degli Studi di Firenze,
Via S.\ Marta 3, I-50139 Florence, Italy -
{\it E-mail address: \tt
massimo.furi@unifi.it}}
\address{Maria Patrizia Pera - Dipartimento di Matematica e Informatica ``Ulisse Dini'',
Universit\`a degli Studi di Firenze,
Via S.\ Marta 3, I-50139 Florence, Italy -
{\it E-mail address: \tt
mpatrizia.pera@unifi.it}}

%%%%%%%%%%%%%%%%%%%%%%%%%%%%%%%%%%%%%%%%%%%%%%%%%%%%%%%%%%%%%%%%%%%%
\begin{abstract}
We study the persistence of eigenvalues and eigenvectors 
of perturbed eigenvalue problems in Hilbert spaces.
We assume that the unperturbed problem has a nontrivial kernel of 
odd dimension and we prove a Rabinowitz-type global continuation result.

The approach is topological, based on a notion of degree for oriented Fredholm maps of index zero between real differentiable Banach manifolds.
\end{abstract}


%\subjclass[2020]{47J10, 47A75, 47H11, 55M25, 34C23}

% dedication
%\dedicatory{Dedicated to the memory of our friend and\\ outstanding mathematician Russell Johnson}

\maketitle

%%%%%%%%%%%%%%%%%%%%%%%%%%%%%%%%%%%%%%%%%%%%%%%%%%%%%%%%%%%%%%%%%%%%
\section{Introduction}
\label{Introduction}

Nonlinear spectral theory is a research field of increasing interest,
which finds application to properties of the structure of the solution set of differential equations, see e.g.\ \cite{ADV,Chi2018}.
In this context a nontrivial question consists in studying nonlinear perturbations of linear problems and in investigating the so-called ``persistence'' of eigenvalues and eigenvectors.

More precisely, let $G$ and $H$ denote two real Hilbert spaces.
By a ``perturbed eigenvalue problem'' we mean a system of the following type:
\begin{equation}
\label{perturbed eigenvalue problem intro}
\left\{
\begin{aligned}
&Lx + s N(x) = \l Cx\\
&x \in \S,
\end{aligned}\right.
\end{equation}
where $s,\l$ are real parameters, $L,C \colon G \to H$ are bounded linear operators, $\S$ denotes the unit sphere of $G$, and $N\colon \S \to H$ is a nonlinear map.
We call \emph{solution} of \eqref{perturbed eigenvalue problem intro} a triple $(s,\l,x) \in \R\per\R\per\S$ satisfying the above system.
The element \mbox{$x \in \S$} is then said a \emph{unit eigenvector} corresponding to the \emph{eigenpair} $(s,\l)$ of \eqref{perturbed eigenvalue problem}, and 
the set of solutions of \eqref{perturbed eigenvalue problem intro} will be denoted by $\Sigma\subseteq \R\per\R\per\S$.

To investigate the topological properties of $\Sigma$,
 we consider \eqref{perturbed eigenvalue problem intro} as a (nonlinear) perturbation of
the eigenvalue problem
\begin{equation}
\label{eigenvalue problem intro}
\left\{
\begin{aligned}
&Lx = \l Cx\\
&x \in \S,
\end{aligned}\right.
\end{equation}
where we assume that the operator $L - \l C \in \L(G,H)$ is invertible for some $\l \in \R$.
When $\l \in \R$ is such that $\Ker(L - \l C)$ is nontrivial, we call $\l$ an \emph{eigenvalue} of the equation $L = \l C$ or, equivalently, of problem~\eqref{eigenvalue problem intro}.
A solution $(\l,x)$ of \eqref{eigenvalue problem intro} will be called an \emph{eigenpoint};
in this case $\l$ and $x$ are, respectively, an eigenvalue and a \emph{unit eigenvector} of the equation $Lx = \l Cx$.

Let $(\l_*,x_*)$ be an eigenpoint of \eqref{eigenvalue problem intro}
and suppose that the following conditions hold:
\begin{itemize}
\item[(H1)]
$C$ is a compact operator,
\item[(H2)]
$\Ker (L-\l_*C)$ is odd dimensional,
\item[(H3)]
$\Im (L-\l_*C) \cap C\left(\Ker (L-\l_*C)\right) = \{0\}$.
\end{itemize}
Under assumptions (H1)--(H3) our main result, Theorem
\ref{main result} below, asserts that
\begin{itemize}
\item
\emph{in the set $\Sigma$ of the solutions of \eqref{perturbed eigenvalue problem intro}, the connected component containing $(0,\l_*,x_*)$ is either unbounded or includes a trivial solution $(0,\l^*,x^*)$ with $\l^* \not= \l_*$.}
\end{itemize}

The proof of Theorem \ref{main result}, which can be thought of as a Rabinowitz-type global continuation result \cite{Ra}, 
 is based on a preliminary study of the ``unperturbed'' problem \eqref{eigenvalue problem intro}.
In particular, notice that the eigenpoints of \eqref{eigenvalue problem intro} coincide with the solutions of the equation
\[
\psi(\l,x) = 0,
\]
where $\psi$ is the $H$-valued function $(\l,x) \mapsto Lx-\l Cx$ defined on the cylinder $\R \per \S$, which is a smooth $1$-codimensional submanifold of the Hilbert space $\R\per G$.
A crucial point is then to evaluate the topological degree of the map $\psi$.
Since the domain of $\psi$ is a manifold, we cannot apply the classical Leary--Schauder degree.
Instead, we use a notion of topological degree for oriented Fredholm maps of index zero between real differentiable Banach manifolds, developed by two authors of this paper, and whose construction and properties are summarized in Section 3 for the reader's convenience.
Such a notion of degree has been introduced in \cite{BeFu1} (see also \cite{BeFu2,BeFu5,BeCaFuPe-s6} for additional details).

Taking advantage of the odd multiplicity assumption (H2), of condition (H1) on the compactness of $C$, and of the transversality condition (H3), we are then able to apply a result of \cite{BeCaFuPe-s5} concerning the case of \emph{simple eigenvalues}.
Precisely, call $\l_* \in \R$ a \emph{simple eigenvalue} of \eqref{eigenvalue problem intro} if there exists $x_* \in \S$ such that $\Ker(L-\l_*C)= \R x_*$ and $H = \Im(L-\l_*C) \oplus \R Cx_*$.
In \cite{BeCaFuPe-s5} we proved that
\begin{itemize}
\item
\emph{if $\l_*$ is a simple eigenvalue of \eqref{eigenvalue problem intro} and $x_*$ and $-x_*$ are the two corresponding unit eigenvectors, then the ``twin'' eigenpoints $p_* = (\l_*,x_*)$ and $\bar p_* = (\l_*,-x_*)$ are isolated zeros of $\psi$.
Moreover, under the assumption that the operator $C$ is compact, they give the same contribution to the \mbox{$bf$-degree,} which is either $1$ or $-1$, depending on the orientation of~$\psi$.}
\end{itemize}
Such an assertion generalizes, to the infinite dimensional case, an analogous result in \cite{BeCaFuPe-s4} concerning a ``classical eigenvalue problem'' in $\R^k$.
Let us point out that the result in \cite{BeCaFuPe-s4} is based on the notion of Brouwer degree for maps between finite dimensional oriented manifolds, whereas, as already stressed, 
 the extension to the infinite-dimensional setting 
of \cite{BeCaFuPe-s5} requires a degree for Fredholm maps of index zero acting between Banach manifolds, as the one introduced in \cite{BeFu1}.
To apply this degree we need the unit sphere $\S$ to be a smooth manifold: for this reason, we restrict our study to Hilbert spaces instead of the more general Banach environment.

\medskip
The study of the local \cite{BeCaFuPe-s1, Chi2017, ChFuPe1, ChFuPe2, ChFuPe3, ChFuPe4,ChFuPe5}
as well as global
\cite{BeCaFuPe-s2,BeCaFuPe-s3,BeCaFuPe-s4,BeCaFuPe-s5, BeCaFuPe-s6}
 persistence property when the eigenvalue $\l_*$ is not necessarily simple
has been performed in recent papers by
the authors, also in collaboration with R.\ Chiappinelli.
In particular a first pioneering result in this sense is due to Chiappinelli \cite{Chi}, who proved the existence of the local persistence of eigenvalues and eigenvectors, in Hilbert spaces, in the case of a simple isolated eigenvalue.

Among others, let us quote our paper \cite{BeCaFuPe-s2} in which we tackled a problem very similar to the one we consider here.
The main result of \cite{BeCaFuPe-s2} regards, roughly speaking, the global persistence property of the \emph{eigenpairs} $(s,\l)$ of \eqref{perturbed eigenvalue problem intro}, in the $s\l$-plane, under the \emph{odd multiplicity} assumption.
Thus, the result we obtain here on the global persistence of the solutions $(s,\l,x)$ of \eqref{perturbed eigenvalue problem intro} was, in some sense, implicitly conjectured in \cite{BeCaFuPe-s2}.

The present paper generalizes the ``global persistence'' property of 
 solution triples which, either in finite-dimensional or infinite-dimensional case, has been studied in
 \cite{BeCaFuPe-s3,BeCaFuPe-s4,BeCaFuPe-s5, BeCaFuPe-s6}
in the case of a simple eigenvalue.
Since it is known that the persistence 
property need not hold if $\l_*$ is an eigenvalue of even multiplicity, it is natural to investigate the odd-multiplicity case.
However such an extension is not trivial and is based on 
advanced degree-theoretical tools.

We close the paper with some illustrating examples showing, in particular, that the odd dimensionality of $\Ker (L-\l_*C)$ cannot be removed, the other assumptions remaining valid.

%%%%%%%%%%%%%%%%%%%%%%%%%%%%%%%%%%%%%%%%%%%%%%%%%%%%%%%%%%%%%%%%%%%%
\section{Preliminaries}
\label{Preliminaries}

In this section we recall some notions that will be used in the sequel.
We mainly summarize some concepts which are needed for the construction of the topological degree for oriented Fredholm maps of index zero between real differentiable Banach manifolds introduced in \cite{BeFu1}, here called $bf$-degree to distinguish it from the Leray--Schauder degree, called $LS$-degree (see \cite{BeFu2,BeFu5,BeCaFuPe-s6} for additional details).

\smallskip
It is necessary to begin by focusing on the preliminary concept of orientation for Fredholm maps of index zero between manifolds.
The starting point is an algebraic notion of orientation for Fredholm linear operators of index zero.

Consider two real Banach spaces $E$ and $F$ and denote by $\L(E,F)$ the space of the bounded linear operators from $E$ into $F$ with the usual operator norm.
If $E=F$, we write $\L(E)$ instead of $\L(E,E)$.
By $\mathrm{Iso}(E,F)$ we mean the subset of $\L(E,F)$ of the invertible operators, and we write $\GL(E)$ instead of $\Iso(E,E)$.
The subspace of $\L(E,F)$ of the compact operators will be denoted by $\K(E,F)$, or simply by $\K(E)$ when $F=E$.
Finally, $\F(E,F)$ will stand for the vector subspace of $\L(E,F)$ of the operators having finite dimensional image (recall that, in the infinite dimensional context, $\F(E,F)$ is not closed in $\L(E,F)$).
We shall write $\F(E)$ when $F=E$.

\medskip
Recall that an operator $T \in \L(E,F)$ is said to be \emph{Fredholm} (see e.g.\ \cite{TaLa}) if its kernel, $\Ker T$, and its cokernel, $\coKer T = F/T(E)$, are both finite dimensional.
The \emph{index} of a Fredholm operator $T$ is the integer
\[
\ind T = \dim(\Ker T) - \dim(\coKer T).
\]
In particular, any invertible linear operator is Fredholm of index zero.
Observe also that, if $T \in \L(\R^k,\R^s)$, then $\ind T = k-s$.

The subset of $\L(E,F)$ of the Fredholm operators will be denoted by $\Phi(E,F)$; while $\Phi_n(E,F)$ will stand for the set $\{T \in \Phi(E,F): \ind T = n\}$.
By $\Phi(E)$ and $\Phi_n(E)$ we will designate, respectively, $\Phi(E,E)$ and $\Phi_n(E,E)$.

\medskip
We recall some important properties of Fredholm operators.

\begin{itemize}
\item
[(F1)] \emph{If $T \in \Phi(E,F)$, then $\Im T$ is closed in $F$.} 
\item
[(F2)] \emph{The composition of Fredholm operators is Fredholm and its index is the sum of the indices of all the composite operators.}
\item
[(F3)] \emph{If $T \in \Phi_n(E,F)$ and $K \in \K(E,F)$, then $T+K \in \Phi_n(E,F)$.}
\item
[(F4)] \emph{For any $n \in \Z$, the set $\Phi_n(E,F)$ is open in $\L(E,F)$.}
\end{itemize}

%\medskip
%One can easily check that $\A(E)$ is a subset of $\Phi_0(E)$.
%This is also a consequence of a well known property regarding Fredholm operators.
%Namely,
%\begin{itemize}
%\item
%[(1)] \emph{if $T \in \Phi_n(E,F)$ and $K \in \F(E,F)$, then $T+K \in \Phi_n(E,F)$.}
%\end{itemize}
%
%Another fundamental property states that
%\begin{itemize}
%\item
%[(2)] \emph{the composition of Fredholm operators is Fredholm and its index is the sum of the indices of all the composite operators.}
%\end{itemize}
%
%An useful consequence of property (2) is the following:
%\begin{itemize}
%\item
%If $T \in \Phi_n(E,F)$ and $k \in \mathbb N$, then the restriction of $T$ to a $k$-codimensional subspace of $E$ is Fredholm of index $n-k$.
%\end{itemize}

\medskip
Let $T \in \L(E)$ be given.
If $I-T \in \F(E)$, where $I \in \L(E)$ is the identity, we say that 
$T$ is an \emph{admissible operator (for the determinant)}.
The symbol $\A(E)$ will stand for the affine subspace of $\L(E)$ of the admissible operators.

It is known (see \cite{Ka}) that the determinant of an operator $T \in \A(E)$ is well defined as follows: $\det T := \det T|_{\hat E}$, where $T|_{\hat E}$ is the restriction (as domain and as codomain) to any finite dimensional subspace $\hat E$ of $E$ containing $\Im(I-T)$, with the understanding that $\det T|_{\hat E} = 1$ if $\hat E = \{0\}$.
As one can check, the function $\det\colon \A(E) \to \R$ inherits most of the properties of the classical determinant.
For more details, see e.g.\ \cite{BeFuPeSp07}.

\medskip
Let $T \in \Phi_0(E,F)$ be given.
As in \cite{BeCaFuPe-s6}, we will say that an operator $K \in \F(E,F)$ is a \emph{companion\footnote{In previous papers, e.g.\ in \cite{BeFu1}, it was used the word
\emph{corrector} instead of \emph{companion}} of $T$} if $T+K$ is invertible.

Observe in particular that any $T \in \Iso(E,F)$ has a \emph{natural companion}: that is, the zero operator $0 \in\L(E,F)$.
This fact was crucial in \cite{BeFu1} for the construction of the $bf$-degree.

Given $T \in \Phi_0(E,F)$, we denote by $\mathcal C(T)$ the (nonempty) subset of $\F(E,F)$ of all the companions of $T$.
The following definition establishes a partition of $\mathcal C(T)$ in two equivalence classes and is a key step for the definition of orientation given in \cite{BeFu1}.

\begin{definition}[Equivalence relation]
\label{equivalence relation}
Two companions $K_1$ and $K_2$ of an operator $T \in \Phi_0(E,F)$ are \emph{equivalent} (more precisely, \emph{$T$-equivalent)} if the admissible operator $(T+K_2)^{-1}(T+K_1)$ has positive determinant.
\end{definition}

\begin{definition}[Orientation]
\label{orientation of T}
An \emph{orientation} of $T \in \Phi_0(E,F)$ is one of the two equivalence classes of $\mathcal C(T)$, denoted by $\mathcal C_+(T)$ and called the class of \emph{positive companions} of the \emph{oriented operator} $T$.
The set $\mathcal C_-(T) = \mathcal C(T) \setminus C_+(T)$ of the \emph{negative companions} is the \emph{opposite orientation of $T$}.
\end{definition}

Some further definitions are in order.

\begin{definition}[Natural orientation]
\label{natural orientation}
Any $T \in \Iso(E,F)$ admits the \emph{natural orientation}: the one given by considering the trivial operator of $\L(E,F)$ as a positive companion.
\end{definition}

\begin{definition}[Oriented composition]
\label{oriented composition}
The \emph{oriented composition} of two oriented operators, $T_1 \in \Phi_0(E_1,E_2)$ and $T_2 \in \Phi_0(E_2,E_3)$, is the operator $T_2T_1$ with the orientation given by considering $K = (T_2+K_2)(T_1+K_1)-T_2T_1$ as a positive companion whenever $K_1$ and $K_2$ are positive companions of $T_1$ and $T_2$, respectively.
\end{definition}

Observe that the oriented composition is associative and, consequently, this notion can be extended to the composition of three (or more) oriented operators.

\begin{definition}[Sign of an oriented operator]
\label{sign}
Let $T \in \Phi_0(E,F)$ be an oriented operator.
Its \emph{sign} is the integer
\[
\sign T =
\left\{
\begin{array}{rl}
+1 & \mbox{if } T \mbox{ is invertible and naturally oriented,}\\
-1 & \mbox{if } T \mbox{ is invertible and not naturally oriented,}\\
 0 & \mbox{if } T \mbox{ is not invertible.}
\end{array}
\right.
\]
\end{definition}

A crucial fact in the definition of oriented map and the consequent construction of the $bf$-degree is that
\begin{itemize}
\item
\emph{the orientation of any operator $T_* \in \Phi_0(E,F)$ induces an orientation of the operators in a neighborhood of $T_*$.}
\end{itemize}
In fact, since $\Iso(E,F)$ is open in $\L(E,F)$, for any companion $K$ of $T_*$ we have that $T+K$ is invertible when $T$ is sufficiently close to $T_*$.
Thus, because of property (F3) of the Fredholm operators, any such $T$ belongs to $\Phi_0(E,F)$.
Consequently, $K$ is as well a companion of $T$.

\begin{definition}
\label{orientation of Gamma}
Let $\G\colon X \to \Phi_0(E,F)$ be a continuous map defined on a metric space $X$.
A \emph{pre-orientation of $\G$} is a function that to any $x \in X$ assigns an orientation $\o(x)$ of $\G(x)$.
A pre-orientation (of $\G$) is an \emph{orientation} if it is \emph{continuous}, in the sense that, given any $x_* \in X$, there exist $K \in \o(x_*)$ and a neighborhood $W$ of $x_*$ such that $K \in \o(x)$ for all $x \in W$.
The map $\G$ is said to be \emph{orientable} if it admits an orientation, and \emph{oriented} if an orientation has been chosen.
In particular, a subset $Y$ of $\Phi_0(E,F)$ is \emph{orientable} or \emph{oriented} if so is the inclusion map $Y \hookrightarrow \Phi_0(E,F)$.
\end{definition}

Observe that the set $\hat\Phi_0(E,F)$ of the oriented operators of $\Phi_0(E,F)$ has a natural topology, and the natural projection $\pi\colon\hat\Phi_0(E,F) \to \Phi_0(E,F)$ is a $2$-fold covering space (see \cite{BeFu2} for details).
Therefore, an orientation of a map $\G$ as in Definition \ref{orientation of Gamma} could be regarded as a lifting $\hat \G$ of $\G$.
This implies that, if the domain $X$ of $\G$ is simply connected and locally path connected, then $\G$ is orientable.

\medskip
Let $f\colon U \to F$ be a $C^1$-map defined on an open subset of $E$, and denote by $df_x \in \L(E,F)$ the Fr\'echet differential of $f$ at a point $x \in U$.

We recall that $f$ is said to be \emph{Fredholm of index $n$}, called \emph{$\Phi_n$-map} and hereafter also denoted by $f \in \Phi_n$, if $df_x \in \Phi_n(E,F)$ for all $x \in U$.
Therefore, if $f \in \Phi_0$, Definition \ref{orientation of Gamma} and the continuity of the differential map $df\colon U \to \Phi_0(E,F)$ suggest the following

\begin{definition}[Orientation of a $\Phi_0$-map in Banach spaces]
\label{Orientation of a map in the flat case}
Let $U$ be an open subset of $E$ and $f\colon U \to F$ a Fredholm map of index zero.
A \emph{pre-orientation} or an \emph{orientation} of $f$ are, respectively, a pre-orientation or an orientation of $df$, according to Definition \ref{orientation of Gamma}.
The map $f$ is said to be \emph{orientable} if it admits an orientation, and \emph{oriented} if an orientation has been chosen.
\end{definition}

\begin{remark}
\label{double nature of a Phi-zero operator}
A very special $\Phi_0$-map is given by an operator $T \in \Phi_0(E,F)$.
Thus, for $T$ there are two different notions of orientations: 
the algebraic one and that in which $T$ is seen as a $C^1$-map, according to Definitions \ref{orientation of T} and \ref{Orientation of a map in the flat case}, respectively.
In each case $T$ admits exactly two orientations (in the second one this is due to the connectedness of the domain $E$).
Hereafter, we shall tacitly assume that the two notions agree.
Namely, $T$ has an algebraic orientation $\o$ if and only if its differential $dT_x\colon \dot x \mapsto T\dot x$ has the $\o$ orientation for all $x \in E$.
\end{remark}

Let us summarize how the notion of orientation can be given for maps acting between real Banach manifolds.
In the sequel, by \emph{manifold} we shall mean, for short, a smooth Banach manifold embedded in a real Banach space.

Given a manifold $\M$ and a point $x \in \M$, the tangent space of $\M$ at $x$ will be denoted by $T_x\M$.
If $\M$ is embedded in a Banach space $\widetilde E$, $T_x\M$ will be identified with a closed subspace of $\widetilde E$, for example by regarding any tangent vector of $T_x\M$ as the derivative $\g'(0)$ of a smooth curve $\g\colon (-1,1) \to \M$ such that $\g(0)=x$.

\medskip
Assume that $f \colon \M \to \N$ is a $C^1$-map between two manifolds, respectively embedded in $\widetilde E$ and $\widetilde F$ and modelled on $E$ and $F$.
As in the flat case, $f$ is said to be \emph{Fredholm of index $n$} (written $f \in \Phi_n$) if so is the differential $df_x \colon T_x\M \to T_{f(x)}\N$, for any $x \in \M$ (see \cite{Smale}).

Given $f \in \Phi_0$, suppose that to any $x \in \M$ it is assigned an orientation $\o(x)$ of $df_x$ (also called \emph{orientation of $f$ at $x$}).
As above, the function $\o$ is called a \emph{pre-orientation} of $f$, and an \emph{orientation} if it is continuous, in a sense to be specified (see Definition \ref{Orientation of a map in the non-flat case}).

\begin{definition}
\label{pre-oriented composition}
The pre-oriented composition of two (or more) pre-oriented maps between manifolds is given by assigning, at any point $x$ of the domain of the composite map, the composition of the orientations (according to Definition \ref{oriented composition}) of the differentials in the chain representing the differential at $x$ of the composite map.
\end{definition}

Assume that $f \colon \M \to \N$ is a $C^1$-diffeomorphism.
Thus, for any $x \in \M$, we may take as $\o(x)$ the natural orientation of $df_x$ (recall Definition \ref{natural orientation}).
This pre-orientation of $f$ turns out to be continuous according to Definition \ref{Orientation of a map in the non-flat case} below (it is, in some sense, constant).
From now on, unless otherwise stated,
\begin{itemize}
\item
\textbf{any diffeomorphism will be considered oriented with the natural orientation}.
\end{itemize}
\noindent
In particular, in a composition of pre-oriented maps, all charts and parametrizations of a manifold will be tacitly assumed to be naturally oriented.

\begin{definition}[Orientation of a $\Phi_0$-map between manifolds]
\label{Orientation of a map in the non-flat case}
Let $f \colon \M \to \N$ be a $\Phi_0$-map between two manifolds modelled on $E$ and $F$, respectively.
A pre-orientation of $f$ is an \emph{orientation} if it is \emph{continuous} in the sense that, given any two charts, $\varphi\colon U \to E$ of $\M$ and $\zeta\colon V \to F$ of $\N$, such that $f(U) \subseteq V$, the pre-oriented composition
\[
\zeta \circ f \circ \varphi^{-1} \colon U \to V
\]
is an oriented map according to Definition \ref{Orientation of a map in the flat case}.

The map $f$ is said to be \emph{orientable} if it admits an orientation, and \emph{oriented} if an orientation has been chosen.
\end{definition}

\medskip
For example any local diffeomorphism $f\colon \M \to \N$ admits the \emph{natural orientation}, given by assigning the natural orientation to the operator $df_x$, for any $x \in \M$ (see Definition \ref{natural orientation}).

In contrast, a very simple example of non-orientable $\Phi_0$-map is given by a constant map from the $2$-dimensional projective space into $\R^2$ (see \cite{BeFu2}).

\begin{notation}
\label{slice and partial map}
Let $D$ be a subset of the product $X \per Y$ of two metric spaces.
Given $x \in X$, we call \emph{$x$-slice of $D$} the set
$D_x = \{y \in Y: (x,y) \in D\}$.
Moreover, if $f \colon D \to Z$ is a map into a metric space $Z$, we denote by $f_x \colon D_x \to Z$ the \emph{partial map of $f$} defined by $f_x = f(x,\cdot)$.
\end{notation}

\medskip
Similarly to the case of a single map, one can define a notion of orientation of a continuous family of $\Phi_0$-maps depending on a parameter $s \in [0,1]$.
To be precise, one has the following

\begin{definition}[Oriented $\Phi_0$-homotopy]
\label{Phi-zero-homotopy}
A \emph{$\Phi_0$-homotopy} between two Banach manifolds $\M$ and $\N$ is a $C^1$-map $h \colon [0,1] \per\M \to \N$ such that, for any $s \in [0,1]$, the partial map $h_s= h(s,\cdot)$ is Fredholm of index zero.
An \emph{orientation} of $h$ is a \emph{continuous function} $\o$ that to any $(s,x) \in [0,1]\per\M$ assigns an orientation $\o(s,x)$ to the differential $d(h_s)_x \in \Phi_0(T_x\M, T_{h(s,x)}\N)$, where ``continuous'' means that, given any chart $\varphi\colon U \to E$ of $\M$, a subinterval $J$ of $[0,1]$, and a chart $\zeta\colon V \to F$ of $\N$ such that $h(J\per U) \subseteq V$, the pre-orientation of the map $\G\colon J\per U \to \Phi_0(E,F)$ that to any $(s,x) \in J\per U$ assigns the pre-oriented composition
\[
d(\zeta \circ h_s \circ \varphi^{-1})_x
= d\zeta_{h(s,x)}d(h_s)_x (d\varphi_x)^{-1}
\]
is an orientation, according to Definition \ref{orientation of Gamma}.

The homotopy $h$ is said to be \emph{orientable} if it admits an orientation, and \emph{oriented} if an orientation has been chosen.
\end{definition}

If a $\Phi_0$-homotopy $h$ has an orientation $\o$, then any partial map $h_s = h(s,\cdot)$ has a \emph{compatible} orientation $\o(s,\cdot)$.
Conversely, one has the following

\begin{proposition}[\!\cite{BeFu1,BeFu2}]
\label{orientation transport}
Let $h\colon [0,1]\per\M \to \N$ be a $\Phi_0$-homotopy, and assume that one of its partial maps, say $h_s$, has an orientation.
Then, there exists and is unique an orientation of $h$ which is compatible with that of $h_s$.
In particular, if two maps from $\M$ to $\N$ are $\Phi_0$-homotopic, then they are both orientable or both non-orientable.
\end{proposition}

As a consequence of Proposition \ref{orientation transport}, one gets that any $C^1$-map $f\colon \M \to \M$ which is $\Phi_0$-homotopic to the identity is orientable, since so is the identity (even when $\M$ is finite dimensional and not orientable).

\medskip
The $bf$-degree, introduced in \cite{BeFu1}, satisfies the three fundamental properties listed below:
\emph{Normalization, Additivity and Homotopy Invariance}.
In \cite{BeFu5}, by means of an axiomatic approach, it is proved that the $bf$-degree is the only possible integer-valued function that satisfies these three properties.

More in detail, the $bf$-degree is defined in a class of \emph{admissible triples}.
Given an oriented $\Phi_0$-map $f\colon \M \to \N$, an open (possibly empty) subset $U$ of $\M$, and a target value $y \in \N$, the triple $(f,U,y)$ is said to be \emph{admissible} for the $bf$-degree provided that $U \cap f^{-1}(y)$ is compact.
From the axiomatic point of view, the $bf$-degree is an integer-valued function, $\deg_{bf}$, defined on the class of all the admissible triples, that satisfies the following three \emph{fundamental properties}.

\medskip
\begin{itemize}
\item
(Normalization) \emph{If $f\colon \M \to \N$ is a naturally oriented diffeomorphism onto an open subset of $\N$, then
\[
\deg_{bf}(f,\M,y) = 1,\quad \forall y \in f(\M).
\]}
\item
(Additivity) \emph{Let $(f,U,y)$ be an admissible triple.
If $U_1$ and $U_2$ are two disjoint open subsets of $U$ such that $U \cap f^{-1}(y) \subseteq U_1 \cup U_2$, then}
\[
\deg_{bf}(f,U,y) = \deg_{bf}(f|_{U_1},U_1,y) + \deg_{bf}(f|_{U_2},U_2,y).
\]
\item
(Homotopy Invariance) \emph{Let $h\colon [0,1]\per\M \to \N$ be an oriented $\Phi_0$-homotopy, and $\g\colon [0,1] \to \N$ a continuous path.
If the set
\[
\big\{(s,x) \in [0,1]\per\M: h(s,x) = \g(s)\big\}
\]
is compact, then
$
\deg_{bf}(h(s,\cdot),\M,\g(s))
$
does not depend on $s \in [0,1]$.}
\end{itemize}

\medskip
Other useful properties are deduced from the fundamental ones (see \cite{BeFu5} for details).
Here we mention some of them.
\medskip
\begin{itemize}
\item
(Localization) \emph{If $(f,U,y)$ is an admissible triple, then
\[
\deg(f,U,y) = \deg(f|_U,U,y).
\]}
\end{itemize}

\medskip
\begin{itemize}
\item
(Existence) \emph{If $(f,U,y)$ is admissible and $\deg_{bf}(f,U,y) \not= 0$, then the equation $f(x) = y$ admits at least one solution in $U$.}
\end{itemize}

\medskip
\begin{itemize}
\item
(Excision) \emph{If $(f,U,y)$ is admissible and $V$ is an open subset of $U$ such that $f^{-1}(y)\cap U \subseteq V$, then
\[
\deg(f,U,y) = \deg(f,V,y).
\]}
\end{itemize}

\medskip
In some sense, given an admissible triple $(f,U,y)$, the integer $\deg_{bf}(f,U,y)$ is an algebraic count of the solutions in $U$ of the equation $f(x) = y$.
In fact, from the fundamental properties one gets the following

\medskip
\begin{itemize}
\item
(Computation Formula)
\emph{If $(f,U,y)$ is admissible and $y$ is a regular value for $f$ in $U$, then the set $U \cap f^{-1}(y)$ is finite and
\[
\deg_{bf}(f,U,y) = \sum_{x \in U \cap f^{-1}(y)} \sign(df_x).
\]}
\end{itemize}

Another useful property that can be deduced from the fundamental ones is the

\medskip
\begin{itemize}
\item
(Topological Invariance)
\emph{If $(f,U,y)$ is admissible and $g\colon \N \to \mathcal O$ is a naturally oriented diffeomorphism onto a manifold $\mathcal O$, then
\[
\deg_{bf}(f,U,y) = \deg_{bf}(g\circ f,U,g(y)).
\]}
\end{itemize}

Some further notation and definitions are in order.

\begin{notation}
Hereafter we will use the shorthand notation $\deg_{bf}(f,U)$ instead of $\deg_{bf}(f,U,0)$, where $f\colon \mathcal M \to F$ is an oriented $\Phi_0$-map from a manifold into a Banach space,
$U$ is an open subset of $\mathcal M$, and $0$ is the null vector of $F$.
Analogously, $\deg_{LS}(f,U)$ means the Leray--Schauder degree $\deg_{LS}(f,U,0)$, where $U$ is an open bounded subset of a Banach space $E$, $f\colon \overline U \to E$ is a compact vector field defined on the closure of $U$, and $0$ is the null vector of $E$.
\end{notation}

\begin{definition}
\label{isolated set}
Let $X$ be a metric space and $\K \subseteq \A \subseteq X$.
We shall say that \emph{$\K$ is an isolated subset of $\A$} if it is compact and relatively open in $\A$.
Thus, there exists an open subset $U$ of $X$ such that $U \cap \A = \K$.
The set $U$ is called an \emph{isolating neighborhood of $\K$ among (the elements of) $\A$}.
\end{definition}

\begin{definition}
\label{definition contribution degree}
Let $f\colon \mathcal M \to F$ be an oriented $\Phi_0$-map from a manifold into a Banach space.
If $\K$ is an isolated subset of $f^{-1}(0)$, we shall call \emph{contribution of $\K$ to the $bf$-degree of $f$} the integer $\deg_{bf}(f,U)$, where $U \subseteq \mathcal M$ is any isolating neighborhood of $\K$ among $f^{-1}(0)$.
The excision property of the degree implies that $\deg_{bf}(f,U)$ does not depend on the isolating neighborhood $U$.
\end{definition}

Regarding Definition \ref{definition contribution degree}, we observe that the finite union of isolated subsets of $f^{-1}(0)$ is still an isolated subset.
Moreover, from the excision and the additivity properties of the \mbox{$bf$-degree} one gets that the contribution to the $bf$-degree of this union is the sum of the single contributions of these subsets.

%%%%%%%%%%%%%%%%%%%%%%%%%%%%%%%%%%%%%%%%%%%%%%%%%%%%%%%%%%%%%%%%%%%%
\section{The eigenvalue problem and the associated topological degree}
\label{Results 1}

Let, hereafter, $G$ and $H$ denote two real Hilbert spaces and consider the eigenvalue problem
\begin{equation}
\label{eigenvalue problem}
\left\{
\begin{aligned}
&Lx = \l Cx\\
&x \in \S,
\end{aligned}\right.
\end{equation}
where $\l$ is a real parameter, $L,C \colon G \to H$ are bounded linear operators, and $\S$ denotes the unit sphere of $G$.
To prevent the problem from being meaningless,
\begin{itemize}
\item
\textbf{we will always assume that the operator $L - \l C \in \L(G,H)$ is invertible for some $\l \in \R$.}
\end{itemize}

When $\l \in \R$ is such that $\Ker(L - \l C)$ is nontrivial, then $\l$ is called an \emph{eigenvalue} of the equation $L = \l C$ or, equivalently, of problem~\eqref{eigenvalue problem}.

A solution $(\l,x)$ of \eqref{eigenvalue problem} will also be called an \emph{eigenpoint}.
In this case $\l$ and $x$ are, respectively, an eigenvalue and a \emph{unit eigenvector} of the equation $Lx = \l Cx$.

Notice that the eigenpoints are the solutions of the equation
\[
\psi(\l,x) = 0,
\]
where $\psi$ is the $H$-valued function $(\l,x) \mapsto Lx-\l Cx$ defined on the cylinder $\R \per \S$, which is a smooth $1$-codimensional submanifold of the Hilbert space $\R\per G$.

By $\mathcal S$ we will denote the set of the eigenpoints of \eqref{eigenvalue problem}.
Therefore, given any $\l \in \R$, the \emph{$\l$-slice} $\mathcal S_\l = \{x \in \S: (\l,x) \in \mathcal S\}$ of $\mathcal S$ coincides with $\S \cap \Ker(L-\l C)$.

Thus, $\mathcal S_\l$ is nonempty if and only if $\l$ is an eigenvalue of problem \eqref{eigenvalue problem}.
In this case $\mathcal S_\l$ will be called the \emph{eigensphere of \eqref{eigenvalue problem} corresponding to $\l$} or, simply, the \emph{\mbox{$\l$-}eigensphere}.
Observe that $\mathcal S_\l$ is a sphere whose dimension equals that of $\Ker(L-\l C)$ minus one.
The nonempty subset $\{\l\} \per \mathcal S_\l$ of the cylinder $\R\per\S$ will be called an \emph{eigenset of \eqref{eigenvalue problem}}.

\begin{remark}
\label{injective}
The assumption that $L - \l C$ is invertible for some $\l \in \R$ implies that, for any $\l \in \R$, the restriction of $C$ to the (possibly trivial) kernel of $L-\l C$ is injective.
\end{remark}

Remark \ref{injective} can be proved arguing by contradiction.
In fact, assume that the assertion is false.
Then, there are $\l_* \in \R$ and a nonzero vector
\[
x_* \in \Ker(L -\l_* C) \cap \Ker C.
\]
This implies that, for any $\l$, the operator $L - \l C$ is non-injective and, consequently, non-invertible, in contrast to the assumption.
In fact, for any $\l$, one has
\[
(L - \l C)x_* = (L - \l_* C)x_* - (\l-\l_*)Cx_* = 0.
\]

\begin{remark}
\label{C compact}
If the operator $C$ is compact, then, from the assumption that $L - \l C$ is invertible for some $\hat \l \in \R$, it follows that $L - \l C$ is Fredholm of index zero for any $\l \in \R$ and, consequently, the set of the eigenvalues of problem \eqref{eigenvalue problem} is discrete.
Moreover, $\Ker(L-\l C)$ is always finite dimensional, and so is the intersection
\[
\Im(L-\l C)\cap C(\Ker(L-\l C)).
\]
Consequently, if this intersection is the singleton $\{0\}$, taking into account Remark \ref{injective} and the fact that $L-\l C \in \Phi_0(G,H)$, one has
\[
H = \Im(L-\l C)\oplus C(\Ker(L-\l C)).
\]
\end{remark}
To prove Remark \ref{C compact} notice that, if $L-\hat\l C$ is invertible, then it is trivially Fredholm of index zero.
Now, given any $\l \in \R$, one has
\[
(L - \l C)= (L - \hat\l C) - (\l-\hat\l)C.
\]
Thus, because of the compactness of $C$, from property (F3) of Fredholm operators, one gets that $L - \l C$ is also Fredholm of index zero.
Finally, the set of the eigenvalues of problem \eqref{eigenvalue problem} is discrete since so is, according to the spectral theory of   linear operators, the set of the characteristic values of $(L - \hat\l C)^{-1}C$.

\medskip
Because of Remark \ref{C compact},
\begin{itemize}
\item
\textbf{from now until the end of this section we assume that the operator $C$ is compact}.
\end{itemize}

\medskip
Observe that the function $\psi$ defined above is the restriction to $\R\per\S$ of the nonlinear smooth map
\[
\overline \psi\colon \R\per G \to H, \quad (\l,x) \mapsto Lx-\l Cx.
\]
According to Remark \ref{C compact}, any partial map $\overline\psi_\l\colon G \to H$ of $\overline \psi$ is Fredholm of index zero.
Since the map $\overline\s\colon \R\per G \to G$ given by $\overline\s(\l,x) = x$ is clearly $\Phi_1$, the same holds true, because of the property (F2) of Fredholm operators, for the composition $\overline\psi = \overline\psi_\l \circ \overline\s$.
Consequently, again because of property (F2), one has that the restriction $\psi$ of $\overline\psi$ to the $1$-codimensional submanifold $\R\per\S$ of $\R\per G$ is $\Phi_0$.

\medskip
Notice that, if $\dim G = 1$, the cylinder $\R\per\S$ is disconnected: it is the union of two horizontal lines, $\R\per\{-1\}$ and $\R\per\{1\}$.
Because of this, to make some statements simpler,
\begin{itemize}
\item
\textbf{from now on, unless otherwise stated, we assume that the dimension of the space $G$ is greater than $1$}.
\end{itemize}
In this case the cylinder $\R\per\S$ is connected, and simply connected if 
$\dim G > 2$.
It is actually contractible if $G$ is infinite dimensional.
Therefore, the $\Phi_0$-map $\psi$, defined above, is orientable and admits exactly two orientations.
We choose one of them~and
\begin{itemize}
\item
\textbf{hereafter we assume that $\psi$ is oriented}.
\end{itemize}

\begin{remark}
\label{compact vector field}
Let $\hat \l \in \R$ be such that $L - \hat\l C$ is invertible and let $Z\colon H \to G$ denote its inverse.
Then, given any $\l \in \R$, the two equations
\begin{itemize}
\item
$\overline\psi_\l(x) = (L - \l C)x = 0 \in H$,
\item
$\overline\eta_\l(x) = Z\overline\psi_\l(x) = (I - (\l - \hat\l)ZC)x = 0 \in G$
\end{itemize}
are equivalent ($I$ being the identity on $G$).
Therefore, if $B$ denotes the unit ball of $G$, the Leray--Schauder degree with target $0 \in G$, $\deg_{LS}(\overline\eta_\l,B)$, of the compact vector field $\overline\eta_\l$ is well defined whenever $\l$ is not an eigenvalue of the equation $Lx = \l Cx$.
\end{remark}

Observe that, as a consequence of the homotopy invariance property of the Leray--Schauder degree, the function $\l \mapsto \deg_{LS}(\overline\eta_\l,B)$ is constant on any interval in which it is defined.
Moreover, in these intervals, $\deg_{LS}(\overline\eta_\l,B)$ is either $1$ or $-1$, since the equation $\overline\eta_\l(x) = 0$ has only one solution: the regular point $0 \in G$.

\begin{remark}
\label{same degree}
Let $U$ be an isolating neighborhood of a compact subset of the set $\mathcal S$ of the eigenpoints of \eqref{eigenvalue problem}, and let $Z\colon H \to G$ be as in Remark \ref{compact vector field}.
Then $\deg_{bf}(\psi,U) = \deg_{bf}(\eta,U)$, provided that the map $\eta = Z\psi$ is the oriented composition obtained by considering $Z$ as a naturally oriented diffeomorphism.
\end{remark}

Concerning possible relations between the $LS$-degree of $\overline\eta_\l$ and the $bf$-degree of $\psi$ (or, equivalently, of $\eta = Z\psi$), we believe that the following is true (but up to now we were unable to prove or disprove).

\begin{conjecture}
\label{conjecture}
Let $[\a,\b]$ be a compact (nontrivial) real interval such that the extremes are not eigenvalues of $Lx = \l Cx$.
Then the $bf$-degree of $\psi$ (or, equivalently, of $\eta = Z\psi$) on the open subset $U = (\a,\b)\per\S$ of\, $\R\per\S$ is different from zero if and only if $\deg_{LS}(\overline\eta_\a,B) \not= \deg_{LS}(\overline\eta_\b, B)$.
\end{conjecture}

In support of the above conjecture we observe that both the conditions \[
\deg_{bf}(\psi,U) \not= 0 \quad \text{and} \quad \deg_{LS}(\overline\eta_\a, B) \not= \deg_{LS}(\overline\eta_\b, B)
\]
imply the existence of at least one eigenpoint $p_* = (\l_*,x_*) \in U$.
The first one because of the existence property of the $bf$-degree and the last one due to the homotopy invariance property of the $LS$-degree.

\begin{definition}
\label{simple}
An eigenpoint $(\l_*,x_*)$ of \eqref{eigenvalue problem} is said to be \emph{simple} provided that the operator $T = L-\l_* C$ is Fredholm of index zero and satisfies the conditions:
\begin{itemize}
\item[(1)]
$\Ker T = \R x_*$,
\item[(2)]
$Cx_*\notin \Im T$.
\end{itemize}
\end{definition}

We point out that, if an eigenpoint $p_*=(\l_*,x_*)$ is simple, then the corresponding eigenset $\{\l_*\}\per\mathcal S_{\l_*}$ is disconnected.
In fact, it has only two elements: $p_*$ and its \emph{twin eigenpoint} $\bar p_*=(\l_*,-x_*)$, which is as well simple.

\medskip
The following theorem obtained in \cite{BeCaFuPe-s6} was essential in the proofs of some results in \cite{BeCaFuPe-s6} concerning perturbations of \eqref{eigenvalue problem}, as problem \eqref{perturbed eigenvalue problem} in the next section.

\begin{theorem}
\label{bf-degree simple eigenpoint}
In addition to the compactness of $C$, assume that $p_* = (\l_*,x_*)$ and $\bar p_*=(\l_*,-x_*)$ are two simple twin eigenpoints of \eqref{eigenvalue problem}.
Then, the contributions of $p$ and $\bar p$ to the $bf$-degree of $\psi$ are equal: they are both either $1$ or $-1$ depending on the orientation of $\psi$.
Consequently, if $U$ is an isolating neighborhood of the eigenset $\{\l_*\} \per \mathcal S_{\l_*}$, one has $\deg_{bf}(\psi,U) = \pm 2$.
\end{theorem}

\smallskip
We close this section strictly devoted to the unperturbed eigenvalue problem \eqref{eigenvalue problem} with a consequence of 
Theorem \ref{bf-degree simple eigenpoint}, which will be crucial in the proof of our main result (Theorem \ref{main result} in Section \ref{Results 2}).

\begin{theorem}
\label{degree theorem}
Let $\l_* \in \R$, put $T = L-\l_*C$, and suppose that
\begin{itemize}
\item[(H1)]
$C$ is a compact operator,
\item[(H2)]
$\Ker T$ is odd dimensional,
\item[(H3)]
$\Im T \cap C(\Ker T) = \{0\}$.
\end{itemize}
Then, given (in $\R\per \S$) an isolating neighborhood $U$ of the eigenset $\{\l_*\}\per \mathcal S_{\l_*}$, one has $\deg_{bf}(\psi,U) \not= 0$.
\end{theorem}
\begin{proof}
Because of the assumption $\Im T \cap C(\Ker T) = \{0\}$, as well as the fact that $T$ is Fredholm of index zero, we can split the spaces $G$ and $H$ as follows:
\[
G = G_1 \oplus G_2\;\; \text{with} \;\;G_1 = (\Ker T)^\perp\;\; \text{and}\;\; G_2 = \Ker T;
\]
\[
H = H_1 \oplus H_2\;\; \text{with} \;\;H_1 = \Im T \;\; \text{and}\;\; H_2 = C(\Ker T).
\]
With these splittings, $T$ and $C$ can be represented in block matrix form as follows:
\[
T =
\left(
\begin{array}{cc}
T_{11} & 0 \\[2ex]
0 & 0 
\end{array}
\right),
\quad C = 
\left(
\begin{array}{cc}
C_{11} & 0 \\[2ex]
C_{21} & C_{22} 
\end{array}
\right).
\]
The operators $T_{11}\colon G_1 \to H_1$ and $C_{22}\colon G_2 \to H_2$ are isomorphisms (the second one because of Remark \ref{injective}), while $C_{11}\colon G_1 \to H_1$ and $C_{21} \colon G_1 \to H_2$ are, respectively, compact and finite dimensional.

We can equivalently regard the equation $\overline\psi(\l,x) = 0$ as $Z\overline\psi(\l,x)=0$, where $Z\colon H \to G$ is an isomorphism.
We choose $Z$ as follows:
\[
Z =
\left(
\begin{array}{cc}
T_{11}^{-1} & 0 \\[2ex]
0 & C_{22}^{-1}
\end{array}
\right).
\]

Given any $\l \in \R$, the operator $\overline \psi_\l = L-\l C \in \L(G,H)$ can be written as $T - (\l-\l_*) C$.
Therefore, putting $\overline\eta = Z \overline \psi\colon \R\per G \to G$, the partial map $\overline\eta_\l\colon G \to G$ (see Notation \ref{slice and partial map}) can be represented as
\[
\overline\eta_\l = 
\left(
\begin{array}{cc}
I_{11} - (\l-\l_*) \widehat C_{11} & 0 \\[2ex]
- (\l-\l_*) \widehat C_{21} & \l_* I_{22}-\l I_{22}
\end{array}
\right),
\]

\noindent
where $I$ is the identity on $G = G_1 \oplus G_2$ and $\widehat C = ZC $ (observe that $\widehat C_{22}$ coincides with the identity $I_{22} \in \L(G_2)$).

This shows that, given any $\l \in \R$, the endomorphism $\overline\eta_\l \colon G \to G$ is a compact vector field.
Therefore, its Leray--Schauder degree on the unit ball $B$ of $G$ is well defined whenever $\l$ is not an eigenvalue of the equation $Lx = \l Cx$, and this happens when $\l$ is close to, but different from, $\l_*$.
Since $G_2$ is odd dimensional and, because of assumption (H3), the geometric and algebraic multiplicities of $\l_*$ coincide, the function $\l \mapsto \deg_{LS}(\overline\eta_\l,B)$ has a sign-jump crossing $\l_*$.
Therefore, if Conjecture \ref{conjecture} were true, we would have done.
So we need to proceed differently.

\smallskip
We consider an isolating neighborhood of the eigenset $\{\l_*\}\per \mathcal S_{\l_*}$ of the type $U = (\a,\b)\per\S$ and we approximate the family of operators $\overline \eta_\l$, $\l \in [\a,\b]$, with a family $\overline \eta_\l^{\,\e} \in \L(G)$, $\l \in [\a,\b]$, having in $(\a,\b)$ only simple eigenvalues; the number of them equal to the dimension of $G_2=\Ker T$.

\smallskip
First of all we point out that

\pallino the operator $I_{11} - (\l-\l_*)\widehat C_{11} \in \L(G_1)$ is invertible for all $\l \in [\a,\b]$,

\smallskip
\noindent
since otherwise the equation $Lx = \l Cx$ would have eigenvalues different from $\l_*$ in the interval $[\a,\b]$.

\smallskip
\noindent
Now, given $\e > 0$ such that $(\l_*-\e,\l_*+\e) \subset (\a,\b)$, we choose a linear operator $A^\e \in \L(G_2)$ with the following properties:

\pallino in the operator norm, the distance between $A^\e$ and $\l_*I_{22}$ is less than $\e$,

\pallino the eigenvalues of $A^\e$ are real and simple,

\pallino any eigenvalue $\l$ of $A^\e$ is such that $|\l - \l_*| < \e$.

\smallskip
\noindent
For any $\l \in \R$ we define $\overline\eta_\l^{\,\e} \in \L(G_1 \oplus G_2)$ by
\[
\overline\eta_\l^{\,\e} = 
\left(
\begin{array}{cc}
I_{11} - (\l-\l_*) \widehat C_{11} & 0 \\[2ex]
- (\l-\l_*) \widehat C_{21} & A^\e-\l I_{22}
\end{array}
\right).
\]
Then, any eigenvalue $\l$ of $A^\e$ is as well an eigenvalue of the equation $\overline \eta_\l^{\,\e}(x) = 0$, and viceversa provided that $\l \in [\a,\b]$.
Therefore, $\overline \eta_\l^{\,\e}(x) = 0$ has exactly $n = \dim(G_2)$ simple eigenvalues in the interval $(\a,\b)$.
Consequently, the function
\[
\eta^\e \colon \R\per\S \to G, \quad (\l,x) \mapsto \overline\eta_\l^{\,\e}(x)
\]
has exactly $n$ eigensets in the open subset $U = (\a,\b)\per\S$ of the cylinder $\R\per\S$, all of them corresponding to a simple eigenvalue.
Therefore, according to Theorem \ref{bf-degree simple eigenpoint}, the contribution of each of them to $\deg_{bf}(\eta^\e,U)$ is either $2$ or $-2$.
Consequently, taking into account that $n$ is odd, one gets $\deg_{bf}(\eta^\e,U) \neq 0$.

Let the isomorphism $Z$ be naturally oriented and let the restriction $\eta$ of $\overline\eta$ to the manifold $\R\per\S$ 
be oriented according to the composition $Z \psi$.
Thus, because of the topological invariance property of the $bf$-degree, we get
\[
\deg_{bf}(\eta,U) = \deg_{bf}(\psi,U).
\]
Hence, it remains to show that, if $\e > 0$ is sufficiently small, then
\[
\deg_{bf}(\eta^\e,U) = \deg_{bf}(\eta,U).
\]
In fact, this is a consequence of the homotopy invariance property of the $bf$-degree.
To see this it is sufficient to show that (if $\e$ is small) the homotopy $h \colon [0,1]\per\overline U \to G$, defined by $h(t,\l,x) = t\eta^\e(\l,x) + (1-t)\eta(\l,x)$, is admissible.
That is,
{\small
\[
h(t,\l,x) \neq 0\; \text{for}\; t \in [0,1]\; \text{and}\; (\l,x) \in \partial U = \{(\l,x) \in [\a,\b]\per\S: \l=\a\; \text{or}\; \l =\b\}.
\]
}
Let us prove that this is true for the left boundary of $U$; that is, for $\l=\a$.
The argument for $\l = \b$ will be the same.

We need to show that (if $\e$ is small) the linear operator $A_t = t\overline\eta_\a^{\,\e} + (1-t)\overline\eta_\a$ of $\L(G)$ is invertible for any $t \in [0,1]$.
In fact, since $A_0 = \overline\eta_\a$ is invertible, and the set of the invertible operators of $\L(G)$ is open, this holds true for all $A_t$ provided that $\e$ is sufficiently small.
\end{proof}

%%%%%%%%%%%%%%%%%%%%%%%%%%%%%%%%%%%%%%%%%%%%%%%%%%%%%%%%%%%%%%%%%%%%
\section{The perturbed eigenvalue problem and global continuation}
\label{Results 2}

Here, as in Section \ref{Results 1}, $G$ and $H$ denote two real Hilbert spaces, $L, C \colon G \to H$ are bounded linear operators, $\S$ is the unit sphere of $G$ and, as in problem \eqref{eigenvalue problem}, the operator $L -\l C$ is invertible for some $\l \in \R$.

Consider the perturbed eigenvalue problem
\begin{equation}
\label{perturbed eigenvalue problem}
\left\{
\begin{aligned}
&Lx + s N(x) = \l Cx\\
&x \in \S,
\end{aligned}\right.
\end{equation}
where $N\colon \S \to H$ is a $C^1$ compact map and $s$ is a real parameter.

\medskip
A \emph{solution} of \eqref{perturbed eigenvalue problem} is a triple $(s,\l,x) \in \R\per\R\per\S$ satisfying \eqref{perturbed eigenvalue problem}.
The element \mbox{$x \in \S$} is a \emph{unit eigenvector} corresponding to the \emph{eigenpair} $(s,\l)$.

The set of solutions of \eqref{perturbed eigenvalue problem} will be denoted by $\Sigma$ and $\mathcal E$ is the subset of $\R^2$ of the eigenpairs.
Notice that $\mathcal E$ is the projection of $\Sigma$ into the $s\l$-plane and the $s=0$ slice $\Sigma_0$ of $\Sigma$ is the same as the set $\mathcal S = \psi^{-1}(0)$ of the eigenpoints of \eqref{eigenvalue problem}, where $\psi$ has been defined in the previous section.

A solution $(s,\l,x)$ of \eqref{perturbed eigenvalue problem} is regarded as \emph{trivial} if $s=0$.
In this case $p=(\l,x)$ is the \emph{corresponding eigenpoint} of problem \eqref{eigenvalue problem}.
When $p$ is simple, the triple $(0,\l,x) \in \Sigma$ will be as well said to be \emph{simple}.
A nonempty subset of $\Sigma$ of the type $\{0\}\per\{\l\}\per\mathcal S_{\l}$ will be called a \emph{solution-sphere}.

We consider the subset $\{(s,\l,x) \in \Sigma: s=0\} = \{0\}\per\Sigma_0 = \{0\}\per\mathcal S$ of the trivial solutions of $\Sigma$ as a \emph{distinguished subset}.
Thus, it makes sense to call a solution $q_* =(0,\l_*,x_*)$ of \eqref{perturbed eigenvalue problem} a \emph{bifurcation point} if any neighborhood of $q_*$ in $\Sigma$ contains nontrivial solutions.

We say that a bifurcation point $q_*=(0,\l_*,x_*)$ is \emph{global} (in the sense of Rabinowitz \cite{Ra}) if in the set of nontrivial solutions there exists a connected component, called \emph{global (bifurcating) branch}, whose closure in $\Sigma$ contains $q_*$ and it is either unbounded or includes a trivial solution $q^*=(0,\l^*,x^*)$ with $\l^* \neq \l_*$.
In the second case $q^*$ is as well a global bifurcation point.

A meaningful case is when a bifurcation point $q_* = (0,\l_*,x_*)$ belongs to a connected solution-sphere $\{0\}\per\{\l_*\}\per\mathcal S_{\l_*}$.
In this case the dimension of $\mathcal S_{\l_*}$ is positive and we will simply say that $x_*$ is a bifurcation point.
In fact, $0$ and $\l_*$ being known, $x_*$ can be regarded as an \emph{alias} of $q_*$.

For a necessary condition as well as some sufficient conditions for a point $x_*$ of a connected eigensphere to be a bifurcation point see \cite{ChFuPe1}.
Other results regarding the existence of bifurcation points belonging to even-dimensional eigenspheres can be found in \cite{BeCaFuPe-s1, BeCaFuPe-s2, BeCaFuPe-s3, BeCaFuPe-s4, BeCaFuPe-s5, BeCaFuPe-s6, ChFuPe2, ChFuPe3, ChFuPe5}.

\medskip
As already pointed out, if the operator $C$ is compact, then $\psi \colon \R\per\S \to H$ is Fredholm of index zero, and this is crucial for the global results regarding the perturbed eigenvalue problem \eqref{perturbed eigenvalue problem}.
Because of this,
\begin{itemize}
\item
\textbf{from now on, unless otherwise stated, we will tacitly assume that the linear operator $C$ is compact}.
\end{itemize}

We define the $C^1$-map
\[
\psi^+\colon \R\per\R\per\S \to H, \quad (s,\l,x) \mapsto \psi(\l,x) + sN(x),
\]
in which $\psi\colon \R\per\S \to H$, as in Section \ref{Results 1}, is given by $\psi(\l,x) = Lx - \l Cx$.
Therefore the set $(\psi^+)^{-1}(0)$ of the zeros of $\psi^+$ coincides with $\Sigma$.

As shown in \cite{BeCaFuPe-s6}, because of the compactness of $C$ and $N$, one gets that
\begin{itemize}
\item
\emph{$\psi^+$ is proper on any bounded and closed subset of its domain.}
\end{itemize}
Consequently, any bounded connected component of $\Sigma$ is compact.
This fact will be useful later.

\medskip
Notice that $\psi^+$ is the restriction to the manifold $\R\per\R\per\S$ of the nonlinear map
\[
\overline\psi^{\,+}\colon \R\per\R\per G \to H, \quad (s,\l,x) \mapsto \overline\psi(\l,x) + s\overline N(x),
\]
where $\overline\psi$ is as in Section \ref{Results 1} and $\overline N$ is the positively homogeneous extension of $N$.

\medskip
The following result of \cite{BeCaFuPe-s6} is crucial for proving the existence of global bifurcation points.

\begin{theorem}
\label{continuation 1}
Given an open subset $\O$ of $\,\R\per\R\per\S$, let
\[
\O_0=\big\{(\l,x) \in \R\per\S: (0,\l,x) \in \O \big\}
\]
be its $0$-slice.
If $\deg_{bf}(\psi,\O_0)$ is well defined and nonzero, then $\O$ contains a connected set of nontrivial solutions whose closure in $\O$ is non-compact and meets at least one trivial solution of \eqref{perturbed eigenvalue problem}.
\end{theorem}

Corollary \ref{compact component} below, which was deduced in \cite{BeCaFuPe-s6} from Theorem \ref{continuation 1}, asserts that the contribution to the $bf$-degree of the $0$-slice of any compact (connected) component of $\Sigma$ is null.
We will need this basic property later.

\begin{corollary}
\label{compact component}
Let $\mathcal D$ be a compact component of $\Sigma$, and let $\mathcal D_0 \subset \R\per\S$ be its (possibly empty) $0$-slice.
Then, if $U \subset \R\per\S$ is an isolating neighborhood of $\mathcal D_0$, one has $\deg_{bf}(\psi,U) = 0$.
\end{corollary}

The following result, obtained in \cite[Theorem 4.5]{BeCaFuPe-s6}, regards the existence of a global branch of solutions emanating from a trivial solution of problem \eqref{perturbed eigenvalue problem} which corresponds to a simple eigenpoint of \eqref{eigenvalue problem}.
\begin{theorem}
\label{main result in s6}
If $(\l_*,x_*)$ is a simple eigenpoint of problem \eqref{eigenvalue problem}, then, in the set $\Sigma$ of the solutions of \eqref{perturbed eigenvalue problem}, the connected component containing $(0,\l_*,x_*)$ is either unbounded or includes a trivial solution $(0,\l^*,x^*)$ with $\l^* \not= \l_*$.
\end{theorem}

We are now ready to prove our main result, which extends Theorem \ref{main result in s6} and provides a global version of Theorem 3.9 in \cite{ChFuPe5}, the latter concerning the existence of local bifurcation points belonging to even dimensional eigenspheres.

\begin{theorem}
\label{main result}
In addition to the compactness of $C$, let $(\l_*,x_*)$ be an eigenpoint of \eqref{eigenvalue problem} and denote by $T$ the non-invertible operator $L-\l_*C$.
Assume that
\begin{itemize}
\item
$\Ker T$ is odd dimensional,
\item
$\Im T \cap C(\Ker T) = \{0\}$.
\end{itemize}
Then, in the set $\Sigma$ of the solutions of \eqref{perturbed eigenvalue problem}, the connected component containing $(0,\l_*,x_*)$ is either unbounded or includes a trivial solution $(0,\l^*,x^*)$ with $\l^* \not= \l_*$.
\end{theorem}
\begin{proof}
Because of the compactness of $C$, according to Remark \ref{C compact}, the operator $L-\l C$ is Fredholm of index zero for all $\l \in \R$.
Moreover, the set of the eigenvalues of problem \eqref{eigenvalue problem} is discrete.
Consequently, the eigenset $\{\l_*\}\per \mathcal S_{\l_*}$, which is compact and nonempty, is relatively open in the set $\mathcal S$ of the eigenpoints.
Thus, it admits an isolating neighborhood $U \subset \R\per\S$ and, therefore, $\deg_{bf}(\psi,U)$ is well defined.

\medskip
Denote by $\mathcal D$ the connected component of $\Sigma$ containing $(0,\l_*,x_*)$.
We may assume that $\mathcal D$ is bounded.
Thus, it is actually compact, since $\psi^+$ is proper on any bounded and closed subset of $\R\per\R\per\S$.
We need to prove that $\mathcal D$ contains a trivial solution $(0,\l^*,x^*)$ with $\l^* \not= \l_*$.

Assume, by contradiction, that this is not the case.
Then the $0$-slice $\mathcal D_0$ of $\mathcal D$ is contained in the eigenset $\{\l_*\}\per \mathcal S_{\l_*}$.
We will show that this contradicts Corollary \ref{compact component}.
We distinguish two cases: $n = 1$ and $n > 1$, where $n$ is the dimension of $\Ker T$.

\smallskip
\emph{Case $n=1$.}
Because of the assumption $\Im T \cap C(\Ker T) = \{0\}$, the eigenpoint $p_* = (\l_*,x_*)$ is simple and $\{\l_*\}\per \mathcal S_{\l_*}$ has only two points: $p_* = (\l_*,x_*)$ and $\bar p_* = (\l_*,-x_*)$.
In this case, according to Theorem \ref{bf-degree simple eigenpoint}, the contribution to the \mbox{$bf$-degree} of any subset of $\{\l_*\}\per \mathcal S_{\l_*}$ is different from zero, and this, having assumed $\mathcal D_0 \subseteq \{\l_*\}\per \mathcal S_{\l_*}$, is incompatible with Corollary \ref{compact component}.

\smallskip
\emph{Case $n>1$.}
The solution-sphere $\{0\}\per\{\l_*\}\per \mathcal S_{\l_*}$ is connected and, consequently, it is contained in the component $\mathcal D$ of $\Sigma$.
Thus, the eigenset $\{\l_*\}\per \mathcal S_{\l_*}$ is contained in the slice $\mathcal D_0$ of $\mathcal D$.
Having assumed $\mathcal D_0 \subseteq \{\l_*\}\per \mathcal S_{\l_*}$, we get $\mathcal D_0 = \{\l_*\}\per \mathcal S_{\l_*}$.
Hence, because of Theorem \ref{degree theorem}, given an isolating neighborhood $U$ of $\mathcal D_0$, one gets $\deg_{bf}(\psi,U) \not= 0$, and we obtain a contradiction with Corollary~\ref{compact component}.
\end{proof}

\begin{remark}
\label{bifurcation point}
Under the notation and assumptions of Theorem \ref{main result} suppose, in addition, that $\dim(\Ker T) > 1$.
Then, the connected component $\mathcal D$ containing $(0,\l_*,x_*)$ contains as well the connected solution-sphere $\mathcal D_* = \{0\}\per\{\l_*\}\per\mathcal S_{\l_*}$.

This implies that there exists at least one point $\hat q = (0,\l_*,\hat x) \in \mathcal D_*$ which is in the closure $\overline {\mathcal D \setminus \mathcal D_*}$ of the difference $\mathcal D \setminus \mathcal D_*$.
Thus, $\hat q$ (or, equivalently, its alias $\hat x \in \mathcal S_{\l_*}$) is a global bifurcation point.
\end{remark}

%%%%%%%%%%%%%%%%%%%%%%%%%%%%%%%%%%%%%%%%%%%%%%%%%%%%%%%%%%%%%%%%%%%%
\section{Some illustrating examples}
\label{Examples}

In this section we provide three examples in $\ell^2$ concerning Theorem \ref{main result}.
The dimensions of $\Ker T$ (where $T = L-\l_*C$) are, respectively, $3$, $2$, and $1$.
The second example, in which $\Ker T$ is two dimensional, shows that in Theorem \ref{main result}, as well as in Remark \ref{bifurcation point}, the hypothesis of the odd dimensionality of $\Ker T$ cannot be removed, the other assumptions remaining valid.

\medskip
Given a positive integer $k$, let $T_k \in \L(\ell^2)$ be the bounded linear operator that to any $x=(\xi_1,\xi_2,\xi_3,\dots) \in \ell^2$ associates the element
\[
T_k x=(0,0,\dots,0,\xi_{k+1},\xi_{k+2},\dots),
\]
in which the first $k$ components are $0$.
Notice that $T_k$ is Fredholm of index zero and its kernel is the $k$-dimensional space
\[
\Ker T_k =\{x\in \ell^2: x=(\xi_1,\xi_2,\dots, \xi_k,0,0,\dots)\},
\]
which is orthogonal to $\Im T_k$.

Hereafter, $C$ will be the well-known compact linear operator defined by
\[
(\xi_1,\xi_2,\xi_3,\dots) \mapsto (\xi_1/1,\xi_2/2,\xi_3/3,\dots,\xi_n/n,\dots).
\]

Given any compact (possibly nonlinear) map $N\colon \ell^2 \to \ell^2$ of class $C^1$, consider the perturbed eigenvalue problem
\begin{equation}
\label{generic example}
\left\{
\begin{aligned}
T_kx + s N(x) & = \l Cx,\\
x & \in \S,
\end{aligned}\right.
\end{equation}
where $\S$ is the unit sphere of $\ell^2$.
As before, we denote by $\Sigma$ the set of solutions $(s,\l,x)$ of \eqref{generic example}.

\medskip
Observe that, for any $k \in \mathbb N$, $k\geq 1$, $\l_* = 0$ is an eigenvalue of the unperturbed equation $T_kx = \l Cx$ and the condition $\Im T_k \cap C(\Ker T_{k}) = \{0\}$ is satisfied.
Therefore, according to Theorem \ref{main result}, given any positive odd integer $k$, any compact perturbing map $N\colon \ell^2 \to \ell^2$ of class $C^1$, and any $x_* \in \S \cap \Ker T_k$, the connected component of $\Sigma$ containing $(0,0,x_*)$ is either unbounded or encounters a trivial solution $(0,\l^*,x^*)$ with $\l^* \neq 0$.

In the three examples below we will check whether or not the assertions of Theorem \ref{main result} and Remark \ref{bifurcation point} hold, by taking, for all of them, the same perturbing map.
Namely,
\[
N\colon \ell^2 \to \ell^2, \quad (\xi_1, \xi_2, \xi_3, \xi_4,\dots) \mapsto(-\xi_2,\xi_1,-\xi_4,\xi_3,0,0,0,\dots).
\]

%\smallskip
\begin{example}[$k=3$]
\label{k=3}
The eigenvalues of the unperturbed equation $T_3x = \l Cx$ are $0,4,5,6,\dots$\;
The first one, $\l_* = 0$, has geometric and algebraic multiplicity 3 and all the other eigenvalues are simple.

A standard computation shows that, in the $s\l$-plane, the set $\mathcal E$ of the eigenpairs has a connected subset $\mathcal E_1$ satisfying the equation $3s^2 + (\l - 2)^2/4 = 1$, corresponding to eigenvectors of the type $(0,0, \xi_3,\xi_4,0,0,...)$.
The set $\mathcal E_1$ is an ellipse with center $(0,2)$ and half-axes $1/\sqrt{3}$ and $2$.
Observe that it includes the eigenpair $(s,\l)=(0,0)$.
All the other eigenpairs are the points of the horizontal lines $\l = 5$, $\l = 6$, $\l = 7$, etc.
Thus, the connected component in $\Sigma$ containing any trivial solution $(0,\l,x)$ with eigenvalue $\l \ge 5$ is unbounded, and this agrees with Theorem \ref{main result}.

The above ellipse can be parametrized by $s = (1/\sqrt{3})\sin\t$, $\l = 2(1-\cos\t)$, $\t \in [0,2\pi]$, and for any $\t$ in the open interval $(0,2\pi)$, the kernel of the equation
\[
T_3x + (1/\sqrt{3})\sin\t Nx - 2(1-\cos\t)Cx = 0
\]
is $1$-dimensional and spanned by the vector
\[
x(\t) = (0,0,(1/\sqrt{3})\sin\t, -(2/3)(1-\cos\t),0,0,\dots).
\]

Since $\mathcal E_1$ is bounded, so is the connected component $\mathcal D$ of $\Sigma$ containing the $2$\mbox{-}dimen\-sional solution-sphere $\mathcal D_* = \{0\}\per\{\l_*\}\per\mathcal S_{\l_*}$ (recall that $\l_*=0$).
As we shall see, $\mathcal D$ includes the twin trivial solutions $(0,\l^*,\pm x^*)$, where
\[
\l^* = 4 \quad \text{and}\quad x^* = x(\pi)/\|x(\pi)\| = (0,0,0,1,0,0\dots).
\]

According to Remark \ref{bifurcation point}, there exists at least one bifurcation point $\hat x \in \mathcal S_{\l_*}$.
Actually, in this case one gets exactly two (global) bifurcation points.
This is due to the fact that $\mathcal D \setminus \mathcal D_*$ has two disjoint ``twin'' branches whose closures meet the solution-sphere $\mathcal D_*$.
The branches can be parametrized with $\t \in (0,2\pi)$ as follows:
\begin{align*}
q(\t) &= \big((1/\sqrt{3})\sin\t, 2(1-\cos\t),x(\t)/\|x(\t)\|\big),\\
\bar q(\t) &= \big((1/\sqrt{3})\sin\t, 2(1-\cos\t),-x(\t)/\|x(\t)\|\big).
\end{align*}
Then, if the following limits exist:
\[
\lim_{\t \to 0}q(\t)\quad \text{and} \quad \lim_{\t \to 0}\bar q(\t),
\]
we get the bifurcation points (as elements of $\mathcal D_*$).
Equivalently, to find the aliases of these points (that is, the corresponding elements in the eigensphere $\mathcal S_{\l_*}$) we compute
\[
\pm \lim_{\t \to 0}(x(\t)/\|x(\t)\|)
\]
obtaining $\pm (0,0,1,0,0,\dots) \in \mathcal S_{\l_*}$.
In fact, to compute the limits, observe that $\sin\t = u_s(\t)\t$ and $2(1-\cos\t) = u_c(\t)\t^2 $, where $u_s$ and $u_c$ are continuous functions such that $u_s(0) = u_c(0) = 1$.
Hence, one quickly obtains
\begin{align*}
\lim_{\t \to 0} \frac{(1/\sqrt{3})\sin\t}{\sqrt{(1/3)\sin^2\t+ (4/9)(1-\cos\t)^2}} &= 1,\\
\lim_{\t \to 0} \frac{-(2/3)(1-\cos\t)}{\sqrt{(1/3)\sin^2\t+ (4/9)(1-\cos\t)^2}} &= 0.
\end{align*}
\end{example}

%\smallskip
\begin{example}[$k=2$]
\label{k=2}
The eigenvalues of the unperturbed equation $T_2x = \l Cx$ are $0,3,4,5,6,\dots$\;
The first one, $\l_*=0$, has geometric and algebraic multiplicity 2 and all the others are simple.

As in Example \ref{k=3}, for any eigenvalue $\l \ge 5$, one gets an horizontal line of eigenpairs containing $(0,\l)$.
Moreover, as one can check, the trivial eigenpairs $(0,3)$ and $(0,4)$ are vertices of an ellipse of eigenpairs with center $(0,7/2)$ and half-axes $1/\sqrt{48}$ and $1/2$, corresponding, as in Example {\ref{k=3}}, to eigenvectors of the type $(0,0, \xi_3,\xi_4,0,0,...)$.
However, in a neighborhood of the origin of the $s\l$-plane there are no eigenpairs, except the isolated one $(0,0)$.
This means that the \emph{solution-circle} $\mathcal D_* = \{0\}\per\{\l_*\}\per\mathcal S_{\l_*}$ is an isolated subset of $\Sigma$.
Therefore, the assertions of Theorem \ref{main result} and Remark \ref{bifurcation point} do not hold in this case.
Moreover, according to Corollary \ref{compact component}, the contribution of $\mathcal D_*$ to the $bf$-degree of the map $\psi$ is zero.

In conclusion, in Theorem \ref{main result} and Remark \ref{bifurcation point}, the assumption that $\Ker T$ is odd dimensional cannot be removed.
\end{example}

%\smallskip
\begin{example}[$k=1$]
\label{k=1}
In this case the eigenvalues of the unperturbed problem are $0,2, 3,4,5,\dots$\;
All of them are simple.
As in the previous two examples, the $s\l$-plane contains infinitely many horizontal lines of eigenpairs.
Their equations are $\l=5$, $\l=6$, $\l=7$,\! \dots\;

In addition to the horizontal lines, the set of the eigenpairs has two bounded components: an ellipse with center $(0,1)$ and half-axes $1/\sqrt{2}$ and $1$, therefore containing $(0,0)$ and $(0,2)$; and, as in Example \ref{k=2}, an ellipse joining $(0,3)$ with $(0,4)$, with center $(0,7/2)$ and half-axes $1/\sqrt{48}$ and $1/2$.

Finally, one can check that, in accord with Theorem \ref{main result}, given any one of the two points of the $0$-dimensional solution-sphere $\{0\}\per\{0\}\per\mathcal S_{0}$, its connected component in $\Sigma$ is bounded and contains a point of $\{0\}\per\{2\}\per\mathcal S_{2}$.
This agrees with Theorem~\ref{main result}.
\end{example}

%\subsection*{Acknowledgement}
%We thank the anonymous referees for their useful remarks.

%%%%%%%%%%%%%%%%%%%%%%%%%%%%%%%%%%%%%%%%%%%%%%%%%%%%%%%%%%%%%%%%%%%%

\end{document}